\documentclass[12pt]{amsart}%
\usepackage[margin=1in]{geometry}
\usepackage{amssymb}
\usepackage{amsfonts}
\usepackage{amsmath}
\usepackage{amsthm}
\usepackage{psfrag, mathrsfs}
\usepackage{graphicx}%
\providecommand{\U}[1]{\protect\rule{.1in}{.1in}}
\newtheorem{theorem}{Theorem}
\newtheorem*{theorem*}{Theorem}
\newtheorem{conjecture}[theorem]{Conjecture}
\newtheorem*{conjecture*}{Conjecture}
\newtheorem*{belief*}{General Belief}
\newtheorem{corollary}[theorem]{Corollary}
\newtheorem{definition}[theorem]{Definition}

\newtheorem{lemma}[theorem]{Lemma}
\newtheorem{notation}[theorem]{Notation}
\newtheorem{proposition}[theorem]{Proposition}

\newtheorem{fact}[theorem]{Fact}

\renewenvironment{proof}[1][Proof]{\noindent\textbf{#1.} }{\ \rule{0.5em}{0.5em}}
\newenvironment{acknowledgements}[1][Acknowledgments]{\noindent\textbf{#1.} }{\ }

\newcommand{\N}{\mathbb{N}}
\newcommand{\Orb}{\mathcal{O}}
\renewcommand{\d}{\mathrm{d}}
\newcommand{\supp}{\mathrm{supp}}
\newcommand{\Lip}{\mathrm{Lip}}

\newcommand{\op}[1]{\left\langle#1\right\rangle}

\newcommand{\var}{\mathrm{var}}
\newcommand{\V}{\mathrm{V}}

\begin{document}

\title[Super-continuous Ergodic Optimization]
{Ergodic Optimization of Super-continuous Functions in the Shift}
\author{Anthony Quas}
\email{aquas(a)uvic.ca}
\author{Jason Siefken}
\email{siefkenj(a)uvic.ca}
\address{Department of Mathematics and Statistics, University of
Victoria, Victoria BC, Canada V8W 3R4}
\thanks{AQ was supported by NSERC; JS was supported by NSERC and the
University of Victoria}

\date{\today}

\begin{abstract}
  Ergodic Optimization is the process of finding invariant probability
  measures that maximize the integral of a given function.  It has
  been conjectured that ``most'' functions are optimized by measures
  supported on a periodic orbit, and it has been proved in several
  separable spaces that an open and dense subset of functions is
  optimized by measures supported on a periodic orbit. All known
  positive results have been for separable spaces. We
  give in this paper the first positive result for a non-separable space, the
  space of \emph{super-continuous} functions on the full shift, where
  the set of functions optimized by periodic orbit measures contains
  an open dense subset.

\end{abstract}

\maketitle

\section{Introduction}
Given an expansive map $T:\Omega\to\Omega$ and a continuous function
$f$, we say that a $T$-invariant probability measure $\mu$
\emph{optimizes} $f$ if
  \[
    \int f\,\d\mu \geq \int f\,\d\nu
  \]
  for all $T$-invariant probability measures $\nu$.  If $y$ is a
  periodic point (i.e., $T^iy=y$ for some $i$), let $\mu_y$ be the
  unique $T$-invariant probability measure supported on $\Orb y$, the
  orbit of $y$.  We call $\mu_y$ a \emph{periodic orbit measure}.  If
  $\mu_y$ optimizes $f$, we will also say that $f$ is optimized by the
  periodic point $y$.

  \begin{belief*}
    ``Most'' functions are optimized by measures supported on a
    periodic orbit.
  \end{belief*}
  ``Most'' can take various meanings, but for our purposes, we
  consider ``most'' to be an open dense set or a residual set.

  \begin{conjecture}
    \label{conj}
    In an expansive dynamical system, the set of Lipschitz functions
    optimized by periodic orbit measures contains an open set that is
    dense in the class of Lipschitz functions.
  \end{conjecture}

  Analogs to Conjecture \ref{conj} have been shown false in the
  general case of continuous functions \cite{jenkinson1}, however they
  have been shown true in a handful of separable spaces. Further,
  various numerical experiments on many important dynamical systems
  support this conjecture (and hint towards some very interesting
  relationships between parameterized families of functions and the
  period of optimizing orbits) \cite{hunt1,hunt2,hunt3}.

  We present a non-separable space where the analog of Conjecture
  \ref{conj} holds true.  Let $\Omega=\mathcal A^\N$ be the one-sided
  shift space on a finite alphabet.  For a sequence $A_n\searrow 0$,
  define a metric $d_A(x,y)=A_n$ if $x$ and $y$ first differ in the $n$th
  place (i.e. $(x)_i=(y)_i$ for $0\le i<n$; $(x)_n\ne (y)_n$).
  Let $C_A(\Omega)$ denote the set
  of Lipschitz functions with respect to the $d_A$ metric, equipped
  with the $d_A$-Lipschitz norm.  If $\{A_n\}$ satisfies the
  additional property that $A_{n+1}/A_{n}\to 0$, we call $f\in
  C_A(\Omega)$ \emph{super continuous}.
  \begin{theorem}
    \label{mainresult}
    Suppose $A=\{A_n\}$ and $A_{n+1}/A_{n}\to 0$.  For a periodic
    orbit measure $\mu_y$ supported on $\Orb y$, let $P_{y}=\{f\in
    C_A(\Omega):\mu_y\text{ is the unique maximizing measure}\}$.
    Then, $\bigcup_{y\text{ periodic}} (P_y)^\circ$ is dense in all of
    $C_A(\Omega)$ under the $A$-norm topology (where $(P_y)^\circ$ is
    the interior of $P_y$).
  \end{theorem}

  We will briefly survey the most well-known positive results.  A
  function $f$ is a Walters function (introduced by Walters in
  \cite{walters2}) if for every $\varepsilon>0$ there exists a
  $\delta>0$ so that for all $n\in \N$ and $x$ and $y$,
  \[
  \max_{0\leq i<n} \{d(T^ix,T^iy)\} \leq \delta \implies
  |S_nf(x)-S_nf(y)|<\varepsilon,
  \]
  where $S_n f(w) = \sum_{i=0}^{n-1} f(T^iw)$.  Bousch shows for
  Walters functions, the analog of Conjecture \ref{conj} holds
  \cite{bousch2}.

  Contreras, Lopes, and Thieullen showed in \cite{lopez} that when
  using a H\"older norm external to a particular union of H\"older
  spaces, the analog of Conjecture \ref{conj} for H\"older spaces
  holds.  Yuan and Hunt made significant progress towards proving
  Conjecture \ref{conj}, though the full result has not yet been
  proved.

  Presented are the already-established theorems for comparison. Note
  that although the theorems are stated in a variety of contexts
  (expanding maps of the circle, one-sided shifts etc.), the essence
  of the problem is present in the simple setting of the one-sided
  shift.

  \begin{theorem*}[Bousch \cite{bousch2}]
    Let $T:X\to X$ be the one-sided shift map and let $W$ denote the
    set of Walters functions on $X$.  If $P\subset W$ is the set of
    Walters functions optimized by measures supported on periodic
    points, then $P$ contains an open set dense in $W$ with respect to
    the Walters norm.
  \end{theorem*}

  \begin{theorem*}[Contreras-Lopes-Thieullen \cite{lopez}]
    Let $T$ be a $C^{1+\alpha}$ expanding map of the circle.  Let
    $H_\beta$ be the set of $\beta$-H\"older functions on $S^1$ and
    let $\mathcal F_{\alpha+}=\bigcup_{\beta>\alpha} H_{\beta}$.  Let
    $P_{\alpha+}\subset \mathcal F_{\alpha+}$ be the subset of
    functions uniquely optimized by measures supported on a periodic
    point.  Then $P_{\alpha+}$ contains a set that is open and dense
    in $\mathcal F_{\alpha+}$ under the $H_{\alpha}$ topology (i.e.,
    the $\alpha$-H\"older norm).
  \end{theorem*}

  \begin{theorem*}[Yuan and Hunt \cite{yuanhunt}]
    Let $T:M\to M$ be an Axiom A map or an expanding map from a
    manifold to itself and let $C_\Lip$ denote the class of Lipschitz
    continuous functions.  For any $f\in C_\Lip$ optimized by a
    measure generated by an aperiodic point, there exists an
    arbitrarily small perturbation of $f$ such that that measure is no
    longer an optimizing measure. Further, any $f\in C_\Lip$ optimized
    by a periodic orbit measure can be perturbed to be stably
    optimized by this periodic orbit measure.
  \end{theorem*}

  With the inclusion of this paper, the current state of the standing
  conjecture is somewhat curious.  Notice that super-continuous
  functions are Lipschitz functions and Lipschitz functions are
  Walters functions.  So, for both a larger and a smaller class than
  Lipschitz functions, analogs of Conjecture \ref{conj} have shown to
  be true, and yet proof of the Lipschitz case remains elusive.

\subsection{Notation \& Definitions}
For some finite alphabet $\mathcal A$, let $\Omega=\mathcal A^\N$ be the
space of one-sided infinite sequences on $\mathcal A$.  For us $\N$
includes $0$.

$T:\Omega\to\Omega$ is the usual shift operator, with $T$-invariant Borel
probability measures on $\Omega$ denoted $\mathcal M$.  We write $\Orb x$
for the orbit of $x$ under $T$, and we say $S$ is a \emph{segment} of
$\Orb x$ if it is an ordered list of the form
$(T^ix,T^{i+1}x,\ldots,T^{i+p-1}x)$ for some $i,p$.  Abusing notation, we
may say $S\subset \Orb x$.

We use $d$ to denote the standard metric on sequences.  That is,
$d(x,y)=2^{-k}$ where $k=\inf\{i:(x)_{i}\neq(y)_{i}\}$ and $(z)_i$ is
the $i$th symbol of $z$.  We follow the convention that
$2^{-\infty}=0$.

  \begin{definition}[Shadowing]
    For two points $x,y$, we say that $x$ \emph{$\varepsilon$-shadows}
    a segment $S=(T^my,\ldots,T^{m+n-1}y) \subset\Orb y$ if
    \[
      d(T^ix,T^{i+m}y)\leq \varepsilon
    \]
    for all $0\leq i<n$.
  \end{definition}
  \begin{definition}[$\varepsilon$-close]
    A point $x$ is said to stay $\varepsilon$-close to a set $Y$ for
    $p$ steps if for all $0\leq i<p$,
    \[
      d(T^ix,Y)\leq \varepsilon.
    \]
  \end{definition}

  \begin{notation}[Ergodic Average]
    For a function $f$ and a point $x$,
    \[
    \op{f}(x) = \lim_{N\to\infty} \frac{1}{N} \sum_{i=0}^{N-1}
    f(T^ix),
    \]
    when the limit exists.
  \end{notation}

  \begin{notation}
    If $x=a_0a_1a_2\cdots$ is a point,
    \[
      (x)_{i}^j = a_{i}a_{i+1}\cdots a_{j-1}a_j
    \]
    is the subword of $x$ from position $i$ to $j$.
  \end{notation}

\section{Summable Variation}
\begin{definition}[Variation]\index{Variation}\index{Notation!$\var_I(f)$}
  The \emph{variation} of a function over level $k$ cylinder sets
  is the maximum a function changes in a distance of $2^{-k}$.  That is,
  if $f$ is a function
  \[
  \var_k(f) = \sup\{|f(x)-f(y)|:d(x,y)\leq 2^{-k}\}.
  \]
\end{definition}

Note that in a shift space, we have additional structure because
distances can only take values of the form $2^{-k}$.

\begin{definition}[Summable Variation]\index{Variation!Summable}
  The function $f$ is of \emph{summable variation} if
  \[
  \sum_{k=0}^\infty \var_k(f) < \infty.
  \]
\end{definition}

\begin{notation}\index{Notation!$\V_k(f)$}
  $\V_k(f)$ represents the tail sum of the variation of $f$ over distances
  smaller than $2^{-k+1}$.  That is
  \[ \V_k(f) = \sum_{j=k}^{\infty} \var_j(f). \]
\end{notation}

Functions of summable variation form a much larger class than Lipschitz
functions. However the general method used in this paper to show Theorem
\ref{mainresult} is to perturb functions by a small multiple of some
canonical ``sharpest'' function.  Yuan and Hunt used this strategy when
dealing with Lipschitz functions by perturbing by $-d(x,\Orb y)$
\cite{yuanhunt}.  But, for functions of summable variation (with the
natural norm of $\|f\|=V_0(f)+\|f\|_\infty$), there is no such
``sharpest'' function. Using the $A$-norms gives us these sharpest
functions again.

We will frequently refer to $A$-metrics and $A$-norms as briefly
introduced earlier.

\begin{definition}[$A$-sequence]\label{def:Anorm}
  An \emph{$A$-sequence}, $(A_n)_{n=0}^\infty$, is a decreasing sequence
  of positive numbers with $A_n\to 0$.

  If there exists $0<\delta<1$ such that $A_{n+1}/A_n<1-\delta$ for
  each $n$, then we say that $(A_n)$ is \emph{lacunary}.
\end{definition}

  Recall that the metric $d_A$ is defined by $d_A(x,y) = A_n$
  if $(x)_i=(y)_i$ for $0\le i<n$ but $(x)_n\ne (y)_n$.

\begin{definition}[$A$-norm]
  If $(A_n)$ is an $A$-sequence, the Lipschitz constant of $f$ is
  $\Lip_A(f)=\sup_k \var_k(f)/A_k$. The $A$-norm is defined by
  $\|f\|_A=\Lip_A(f)+\|f\|_\infty$.
\end{definition}

Of course if $A$ is the sequence $(2^{-n})_{n=0}^\infty$, we recover the
standard distance and Lipschitz norm. We write the set of Lipschitz
functions with respect to $d_A$ as $C_A(\Omega)$ or simply $C_A$.

Notice that since $A$ satisfies $A_n\to 0$, $C_A(\Omega)\subset
C(\Omega)$ is a subset of the continuous functions on $\Omega$. Further,
$C_A$ is a non-separable Banach space as the functions
$f_x(\cdot)=d(x,\cdot)$ for $x\in\Omega$ are an uncountable uniformly
discrete set.

\section{Preliminary lemmas}
We will first establish several results that do not depend on super
continuity.

\begin{definition}[In Order for One Step]
  For points $x,y$, let $S=(T^jy,T^{j+1}y,\ldots,$ $T^{j+k}y)\subset \Orb y$,
  and suppose that there is a unique closest point $y'\in S$
  to $x$.  That is,
  \[
  d(x,y')< d(x,S\backslash\{y'\}).
  \] We say that $x$
  \emph{follows $S$ in order for one step} if
  $Ty'\in S$ and $Ty'$ is the unique closest point to $Tx$.  That is $Ty'\in S$ and
  \[
  d(Tx,Ty')<d(Tx,S\backslash\{Ty'\}).
  \]
\end{definition}
\begin{definition}[In Order]
  For some point $y$, let $S=(T^jy,T^{j+1}y,\ldots,$ $T^{j+k}y)\subset \Orb y$.
  For some point $x$, we say that \emph{$x$ follows $S$ in order} for $p$ steps
  if $x,Tx,\ldots, T^{p-1}x$ each follow $S$ in order for one step.
\end{definition}
Following in order is very similar to the concept of shadowing except
that the distance requirement in shadowing is replaced by a uniqueness
requirement.  The following In Order Lemma is due to Yuan and Hunt
\cite{yuanhunt}.

\begin{lemma}[In Order Lemma]\index{In Order Lemma}
  \label{inorderlemma}
  Let $y$ be a periodic point of period $p$, and let
  \[
  \rho \le \min_{0\leq i< j<p} d(T^iy,T^jy)/4.
  \]
  For any point $x$, if $x$ stays $\rho$-close to $\Orb y$ for $k+1$
  steps, then $x$ follows $\Orb y$ in order for $k$ steps.  That is,
  there exists some $i'$ such that for $0\leq j\le
  k$, \[d(T^{j}x,T^{i'+j}y)\leq\rho.\]
\end{lemma}
\begin{proof}
  Let $\gamma=\min_{0\leq i< j<p}d(T^iy,T^jy)$. We first
  derive a fact about the shift space due to its
  ultrametric properties.  Suppose $y',y''\in \Orb y$ and for some
  point $x$, $d(x,y'),d(x,y'')\leq \gamma/2$.  By the ultrametric
  triangle inequality we have
  \begin{equation}
    \label{equ:inorderlemma}
    d(y',y'') \leq \max(d(x,y'),d(x,y'')) \leq \gamma/2.
  \end{equation}
  Since $\gamma$ was the smallest distance between points in $\Orb y$,
  equation (\ref{equ:inorderlemma}) gives $y'=y''$.  This shows that
  for any point $x$, if $d(x,\Orb y)\leq \gamma/2$, then there is a
  unique closest point in $\Orb y$ to $x$.

  Let $x$ be a point that stays $\rho$-close to $\Orb y$ for $k+1$ steps.
  By definition, we have
  \[
  d(x,\Orb y) \leq \rho \leq \gamma/4.
  \]
  Since $\gamma$ is the minimum distance between points in $\Orb y$,
  there is a unique $i'$ such that
  \[
  d(x,T^{i'}y) \leq \rho.
  \]
  We then have that
  \[
  d(Tx,T^{i'+1}y)\leq 2\rho \leq \gamma/2,
  \]
  and so $T^{i'+1}y$ is the unique closest point to $Tx$.  Thus, $x$
  follows $\Orb y$ in order for one step.  But, by assumption we have
  $d(Tx,\Orb y)\leq \rho$, so $d(Tx,\Orb y)=d(Tx,T^{i'+1}y)$ gives us
  that $Tx$ follows $\Orb y$ in order for one step and so $x$ follows
  $\Orb y$ in order for two steps.  Continuing by induction, we see
  that $x$ follows $\Orb y$ in order for $k$ steps; that is
  \[
  d(T^{j}x,T^{i'+j}y)\leq \rho
  \qquad \text{for}\quad 0\leq j\le k.
  \]

\end{proof}

\begin{lemma}[Shadowing Lemma]
  \label{shadowlemma}
  For a point $y$, let $S=(T^iy,T^{i+1}y,\ldots,T^{i+k-1}y)$ be a
  segment of $\Orb y$.  For any $\rho<1$, if a point $x$
  $\rho$-shadows $S$ for $k$ steps, the distance from $T^jx$ to  $S$ for $0\leq
  j<k$ is bounded by
  \[
  d(T^{j}x,T^{i+j}y)\leq\rho 2^{-((k-1)-j)}.
  \]
\end{lemma}
\begin{proof}
  Let $l=\inf\{w:2^{-w}\leq \rho\}$ and note $\rho <1$ implies $l\geq
  1$.  Since $x$ $\rho$-shadows $S$ for $k$ steps, we have
  $(T^{j}x)_0^{l-1}=(T^{i+j}y)_0^{l-1}$ for $0\le j\le k-1$, and so
  $(x)_{0}^{k+l-2}=(T^iy)_{0}^{k+l-2}$, which gives the result.
\end{proof}

\begin{lemma}[Parallel Orbit Lemma]
  \label{parallelorbitlemma}
  For a function of summable variation $f$, if $T^mx$ $2^{-r}$-shadows
  $\Orb y$ for $k$ steps (i.e., there exists $i$ so
  $d(T^{m+j}x,T^{i+j}y)\leq 2^{-r}$ for $0\leq j< k$), then for $r>0$,
  \[
  \sum_{j=0}^{k-1} \left|f(T^{m+j}x)-f(T^{i+j}y)\right|
  \leq \V_{r}(f).
  \]
\end{lemma}
\begin{proof}
  Suppose $x,y$ are points such that $d(T^{m+j}x,T^{i+j}y)\le 2^{-r}$
  where $r\ge 1$ for $0\le j<k$. The Shadowing Lemma (Lemma
  \ref{shadowlemma}) gives us that
  \[
  d(T^{m+j}x,T^{i+j}y)\leq 2^{-(r+(k-1)-j)}.
  \]
  We then have
  \[
  \sum_{j=0}^{k-1} |f(T^{m+j}x)-f(T^{i+j}y)|
  \leq \sum_{j=r}^{r+k-1} \var_j(f) \leq \V_{r}(f).
  \]
\end{proof}

\section{Ma\~n\'e-Conze-Guivarc'h normal form and main result}

Heuristically, let us consider the following: Suppose $f$ is optimized
by $\mu_{\max}$ and $\int f\d\mu_{\max} = 0$. We will define a
function $f^*$ to represent the ``payoff of going backwards to
infinity.''  Before we describe what $f^*$ means, let us consider the
payoff of going backwards a finite number of steps. For a point $x$,
there is some point $a_1^1x\in T^{-1}x$ such that $f(a_1^1x)\geq
f(b_1x)$ for any symbol $b_1$.  In other words, $a_1^1x$ is a maximal
one-step backwards extension of $x$.  Continuing, there is some point
$a_2^2a_1^2x\in T^{-2}x$ so that $f(a_2^2 a_1^2x) + f(a_1^2)\geq f(b_2
b_1x)+f(b_1x)$ for any word $b_2 b_1$, making $a_2^2a_1^2x$ a maximal
two-step backwards extension of $x$.  It is important to note that the
symbol $a_1^2$ need not be the same as the symbol $a_1^1$, and so it
is in no way immediate that there should be some convergent way to
pick an infinite maximal backwards extension of $x$.

However, ignoring these issues for the moment, one can imagine that
$n$-step backwards extensions of $x$ look more and more like generic
points of $\mu_{\max}$ (if $\mu_{\max}$ is a periodic orbit measure,
this should be especially plausible).  We now informally define $f^*$
as
\[
f^*(x) = f(a_1^\infty x) + f(a_2^\infty a_1^\infty x) + f(a_3^\infty
a_2^\infty a_1^\infty x) + \cdots,
\]
where $\cdots a_3^\infty a_2^\infty a_1^\infty x$ is an infinite
maximal backwards extension of $x$.  Since $\int f\d\mu_{\max}=0$, it
is reasonable to expect that if $f^*$ converges, it is bounded above.
Ignoring any issues of convergence, consider
\[
f^*\circ T-f^*.
\]
Suppose $x=x_0x_1\cdots$ is a point with maximal backwards extension
$\cdots a_2a_1x_0x_1\cdots$. We immediately see $(f^*\circ T-f^*)(x)
\geq f(x)$, since either the maximal backwards extension of
$Tx=x_1x_2\cdots$ is $\cdots a_2a_1x_0x_1\cdots$, which would give us
$(f^*\circ T-f^*)(x) = f(x)$, or there is an alternative backwards
extension of $Tx$ that yields a bigger payoff than $\cdots
a_2a_1x_0x_1\cdots$ and so $(f^*\circ T-f^*)(x) > f(x)$.

Since $f^*\circ T-f^*$ is a co-boundary (a function of the form
$h-h\circ T$) and so integrates to zero with respect to any invariant
measure, the function $\hat f = f-(f^*\circ T-f^*)$ is co-homologous
to $f$ (and so $\int f\d\mu = \int \hat f\d\mu$ for all invariant
measures $\mu$), with the added property that $\hat f\leq 0$.

The Ma\~n\'e-Conze-Guivarc'h procedure is a way of producing a well
defined $f^*$.  We use a method due to Bousch \cite{bousch}, which
produces $f^*$ as a fixed point of an operator that reflects the idea of
a maximal backwards extension.

For $f\in C_A$, define the operator $\Phi_f:C_A\to C_A$ by
\[
(\Phi_fg)(x)=\max_{y\in T^{-1}x} \{ (f+g)(y) \}.
\]

\begin{proposition}[Bousch]
  \label{fixedpoint}
  Let $(A_n)$ be a lacunary $A$-sequence. For a fixed function $f\in C_A$
  with $\sup_{\mu\in\mathcal M}\int f\d\mu=0$, the operator $\Phi_f$
  as defined above has a fixed point.
\end{proposition}

The proof follows standard lines with minor adaptations for the case of
$A$-norms rather than Lipschitz norms. We briefly summarize the steps,
referring the reader to Bousch \cite{bousch} for more details.

\begin{proof}[Proof sketch]
  Let $A_{n+1}/A_{n}<1-\delta$ for all $n$ (where $0<\delta<1$).
  We claim that $\Phi_f$ maps $C=\{g\colon\Lip_A(g)
   \le \Lip_A(f)/\delta\}$ into itself. We do part of this step in detail
   since we need a fact from it later. Let $g\in C$
   and let $x$ and $x'$ differ first in
   their $(n-1)$st coordinates. Using the notation $ix$ to
   denote the sequence with its first symbol defined by $(ix)_0=i$ and
   all remaining symbols defined by $(ix)_{k+1}=x_k$, we have
   \begin{align*}
     \Phi_f(g)(x)-\Phi_f(g)(x')&=\max_i(f(ix)+g(ix))-\max_j(f(jx')+g(jx'))\\
     &\le \max_i(f(ix)+g(ix)-f(ix')-g(ix'))\\
     &\le \var_{n}(f)+\var_{n}(g)
   \end{align*}
   By symmetry we deduce
   \begin{equation}\label{eq:bouschfactoid}
   \var_{n-1}(\Phi_f(g))\le \var_n(f)+\var_n(g).
   \end{equation}
   Straightforward manipulation then shows that $\Phi_f(g)\in C$.

   Taking a quotient of $C$
   by the relation $\sim$ where two functions $g$ and $g'$ are related
   if they differ by a constant,
   one obtains a compact (with respect to the quotient of the supremum
   norm topology) convex set $C/\sim$
   on which $\Phi_f$ acts continuously. Hence, there is a fixed point.
   This fixed point corresponds to a function $h\in C$ such that $\Phi_f(h)=h+\beta$ for
   some constant $\beta$. One then shows that $\sup_{\mu\in\mathcal M}
   \int f\,d\mu=0$ implies $\beta=0$
\end{proof}

\begin{theorem}
  \label{cohom}
  Let $(A_n)$ be a lacunary $A$-sequence.
  There exists a constant $\gamma_A>1$, dependent only on the choice of $A$-sequence, such that for all $f\in C_A$ with
  $\sup_{\mu\in \mathcal M}\int f\d\mu = 0$,
  there exists a co-homologous function $\hat f$ with $\hat f\leq 0$
  and
  \[
  \|\hat f\|_A \leq \gamma_A \|f\|_A \qquad \V_n\hat f \leq \gamma_A
  \|f\|_A A_n.
  \]
\end{theorem}
\begin{proof}
  Suppose $A_{n+1}/A_{n}<1-\delta$ for all $n$ (for some $0<\delta<1$).
  By Proposition \ref{fixedpoint}, we may find $h$, a
  fixed point of $\Phi_f$ with
  \[
  \|h\|_A \leq \frac{\text{Lip}_A(f)}{\delta} + \|h\|_\infty
  \leq (A_0+1)\frac{\|f\|_A}{\delta}.
  \]
  However, from \eqref{eq:bouschfactoid} we have
  \begin{equation*}
    \label{eq:h_var_ineq}
    \var_{n-1}(h) = \var_{n-1}(\Phi_fh) \leq \var_{n}(f)+\var_{n}(h).
  \end{equation*}
  This gives
  \[
  \frac{\var_n(h\circ T)}{A_n} \leq \frac{\var_{n-1} (h) }{A_n}
  \leq \frac{\var_{n} (f) + \var_{n}(h)}{A_n},
  \]
  and so $\|h\circ T\|_A\leq \|f\|_A+\|h\|_A$.  Let $\hat f =
  f+h-h\circ T$.  $\hat f$ has the desired properties and
  \[
  \|\hat f\|_A \leq \|f\|_A+\|h\|_A+\|h\circ T\|_A \leq 2\|f\|_A+2\|h\|_A
  \leq
  \frac{2(A_0+1+\delta)}{\delta}\|f\|_A.
  \]

  Let us now focus on finding a constant such that $\V_n\hat f \leq K
  \|f\|_A A_n$.  From our bound on $\|\hat f\|_A$, we know $\var_k
  \hat f \leq \frac{2(A_0+1+\delta)}{\delta}\|f\|_A A_k$.
  $A_{k+1}/A_{k} < 1-\delta$ for all $k$ gives that $\sum_{k\ge n} A_k \leq
  A_n/\delta$ and so
  \[
  \V_n \hat f \leq
  \frac{2(A_0+1+\delta)}{\delta^2}\|f\|_A A_n
  \]
  Letting
  $\gamma_A=2(A_0+1+\delta)/\delta^2$
  completes the proof.
\end{proof}

It should be noted that Theorem \ref{cohom} can trivially be applied
to functions $f$ where $\sup_{\mu\in \mathcal M}\int
f\d\mu=\beta\neq0$ by letting $\hat f = \widehat{f-\beta}+\beta$.

\begin{corollary}
  \label{cor:cohom}
  Theorem \ref{cohom} holds with the weakened assumption that $\limsup A_{n+1}/A_{n} < 1$.
\end{corollary}
\begin{proof}
  Since $\limsup A_{n+1}/A_{n} < 1$, we can construct a
  sequence $B_n$ such that $B_{n+1}/B_{n} < 1-\delta$ for some
  $0<\delta<1$ and
  $B_i=A_i$ for $i>N$ for some finite $N$.  Since we only changed a
  finite number of terms of $A$ to produce $B$, $\|\cdot\|_A$ and
  $\|\cdot\|_B$ are equivalent.  Let $M$ be such that $\|f\|_A \leq M
  \|f\|_B$ for all $f\in C_A$ and $M'=\max A_n/B_n$.
  Letting $\gamma_A = MM'\gamma_B$ completes the proof.
\end{proof}

Though not dependent on Theorem \ref{cohom}, it is convenient to note
that $\gamma_A$ from Theorem \ref{cohom} also bounds $\V_n f$ in the
expected way.
\begin{fact}
  \label{cor:Vnbound}
  If $(A_n)$ is a lacunary $A$-sequence, then for $f\in C_A$
  \[
  \V_n f \leq \gamma_A \|f\|_A A_n,
  \]
  where $\gamma_A$ is as in Theorem \ref{cohom}.
\end{fact}

We now have machinery in place to give a quick proof of Proposition
\ref{shadowlemmab}, which establishes a relationship between the number
of points in the support of a periodic orbit measure and how close such
measures come to optimizing a fixed function.  This result was first
established by Yuan and Hunt (without using the Ma\~n\'e-Conze-Guivarc'h
Lemma) in \cite{yuanhunt} for Lipschitz functions.

\begin{proposition}[Yuan and Hunt]\label{shadowlemmab}
  Let $(A_n)$ be a lacunary $A$-sequence. Let $f\in C_A$ and $x$ be an
  optimal orbit for $f$ (i.e., a typical point of a maximizing
  measure). Let $y$ be a point of period $p$, and $r>0$.  If a segment of
  $\Orb x$ $2^{-r}$-shadows $\Orb y$ for one period (i.e., there exist
  $m, m'$ such that $d(T^{i+m}x,T^{i+m'}y)\leq 2^{-r}$ for $0\leq
  i<p$), then
  \[
  \op{f}(x) - \gamma_A\|f\|_A A_r / p \leq \op{f}(y) \leq \op{f}(x),
  \]
  where $\gamma_A$ is as in Theorem \ref{cohom}.
\end{proposition}
\begin{proof}
  Let $y$ be a period $p$ point with the property that a segment of
  $\Orb x$ $2^{-r}$-shadows $\Orb y$ for $p$ steps.  By renaming some
  $T^jy$ as $y$, without loss of generality we may assume that a
  segment of $\Orb x$ $2^{-r}$-shadows $y$. That is, there exists some
  $m$ so that $d(T^{m+i}x,T^iy)\leq 2^{-r}$ for $0\leq i<p$.  Let
  $x'=T^mx$.

  By Theorem \ref{cohom}, we may find $\hat f$ co-homologous to $f$
  with $\hat f(\Orb x)=\hat f(\Orb x')=\op{f}(x)$.  Since for $0 < i
  \leq p$ we have
  \[
  d(T^ix',T^iy) \leq 2^{-(r+(p-1-i))},
  \]
  we may apply the Parallel Orbit Lemma (\ref{parallelorbitlemma}) to get
  \[
  \left |\sum_{i=0}^{p-1}  \left(\hat f(T^ix') - \hat
      f(T^iy)\right) \right|
  =\left |\left(\sum_{i=0}^{p-1}  \hat f(T^ix')\right) - p\op{\hat
      f}(y) \right|
  \leq \V_r \hat f.
  \]
  The proposition follows from the fact that $\hat
  f(T^ix')=\op{f}(x)$ and that by Theorem \ref{cohom} $\V_r\hat f
  \leq \gamma_A \|f\|_A A_r$.
\end{proof}

Using methods similar to those in Yuan and Hunt\cite{yuanhunt}, one can
show that Proposition \ref{shadowlemmab} holds for any function $f$ of
summable variation, and one can produce a slightly stronger bound of
$\op{f}(x) - 4V_r f / p \leq \op{f}(y) \leq \op{f}(x)$.

We are now ready to prove Theorem \ref{mainresult} by using
$d_A(\cdot, \Orb y)$ as a ``sharpest'' function that will penalize any
measure that gives mass to $(\Orb y)^c$.

\begin{theorem*}[Theorem \ref{mainresult}]
  Let $(A_n)$ be an $A$-sequence satisfying $A_{n+1}/A_{n}\to 0$.  For a
  periodic orbit measure $\mu_y$ supported on $\Orb y$, let
  $P_{y}=\{f\in C_A(\Omega):\mu_y\text{ is the unique maximizing
    measure}\}$.  Then, $\bigcup_{y\text{ periodic}} (P_y)^\circ$ is
  dense in $C_A(\Omega)$ (where $(P_y)^\circ$ is the interior of
  $P_y$).
\end{theorem*}

\begin{proof}
  We will show that for any function $f$, there exists an arbitrarily
  small perturbation, $\tilde f$, of $f$ and a periodic orbit measure
  $\mu_y$, such that all functions in an open neighbourhood of $\tilde
  f$ are uniquely optimized by $\mu_y$.

  Since $\liminf A_{n+1}/A_{n} = 0$, by Corollary
  \ref{cor:cohom}, passing to an equivalent norm if necessary, we
  may assume $A_{n+1}/A_{n}\leq 1/2$ for all $n$.
  Fix $f\in C_A$ and let
  $\mu_{\max}$ be an optimizing measure for $f$.  Fix
  $x\in\supp(\mu_{\max})$.  Without loss of generality, assume
  $\op{f}(x) = 0$ and let $\hat f$ be co-homologous to $f$ with $\hat
  f \leq 0$.

  Suppose we showed that an arbitrarily small perturbation $\hat f+g$
  of $\hat f$ were such that the open ball of radius $\varepsilon$
  about $\hat f+g$ is uniquely optimized by a periodic orbit measure
  $\mu_y$. Since $\hat f$ and $f$ are co-homologous, this means that
  $f+g$ is uniquely optimized by $\mu_y$ and in fact the open ball of
  radius $\varepsilon$ about $f+g$ is uniquely optimized by $\mu_y$.
  Thus, it is sufficient to only consider small perturbations of $\hat
  f$.

  Fix $0<\varepsilon<1$.  For a fixed $k$ (to be determined later),
  find a minimal recurrence in $x$ of a block of $k$ symbols.  That
  is, find $i<j$ such that $d(T^ix,T^jx) \le 2^{-k}$ but for $i\leq
  i'<j'<j$, we have $d(T^{i'}x, T^{j'}x)> 2^{-k}$. Notice that such
  a minimal recurrence exists for all $k$ by the pigeonhole principle.

  Let $p=j-i$ and let $y$ be the point of period $p$ satisfying
  $(y)_i^{j-1}=(x)_i^{j-1}$.
  Since $d(T^ix,T^jx)\le 2^{-k}$ we see that $(y)_i^{j+k-1}=(x)_i^{j+k-1}$.
  It follows that the orbit segment $(T^ix,\ldots,T^{j-1}x)$
  $2^{-(k+1)}$-shadows $T^iy$.

  Let $2^{-l}=\min_{i\leq i'<j'<j}\{d(T^{i'}y,T^{j'}y)\}$ be the minimum
  distance between points in $\Orb y$ and notice that by construction of $y$
  and the ultrametric property, $2^{-l} \geq 2^{-(k-1)}$.

  Define the perturbation function $g$ by $g(t) = -d_A(t,\Orb y)$, and
  let $ \tilde f = \hat f-\varepsilon g.  $

  We will now show that provided $k$ is sufficiently large, the
  measure supported on $\Orb y$ is the unique optimizing measure for
  functions lying in a $\|\cdot\|_A$-open ball about $\tilde f$.

  Let $Q=\{\tilde f + h: \|h\|_A < \varepsilon\sigma\}$ with $\sigma <
  1$ to be determined later.  Fix $\hat f-\varepsilon g +h \in Q$ and
  let $q$ be its normalization, $q=\hat f-\varepsilon g + h + \beta$ where $\beta =
  -\sup_{\mu\in\mathcal M}\int (\hat f-\varepsilon g +h) \d\mu$.

  Let $\gamma_A$ be as in Theorem \ref{cohom}. Recall that $\gamma_A>1$.
  We then have $\V_n
  \hat f \leq \gamma_A \|f\|_A A_n$.  Further, since
  $\varepsilon,\sigma < 1$, Fact \ref{cor:Vnbound} gives us
  $\V_n(\varepsilon g), \V_n h \leq \gamma_A A_n$.  Let
  $L=\gamma_A^2(\|f\|_A + 2)$.  Since $\V_n(\hat f-\varepsilon
  g+h)=\V_nq$ we have
  \[
  \V_n \hat f, \V_n \tilde f, \V_n q \leq L A_n \qquad
  \text{and}\qquad \gamma_A \V_n f \leq L A_n,
  \]
  with the second inequality following from Fact \ref{cor:Vnbound}.
  Further, $L$ only depends on $A$ and $\|f\|_A$.

  Since $x$ $2^{-(k+1)}$-shadows $\Orb y$ for $p$ steps, we can get a
  good bound for $\beta$.  By construction
  \[
  \op{q}(y)  = \op{f}(y)-\varepsilon \op{g}(y) + \op{h}(y) + \beta \leq 0,
  \]
  and so
  \[
  \beta \leq - \op{f}(y) + \varepsilon \op{g}(y) - \op{h}(y) = -\op{f}(y)-\op{h}(y).
  \]
  Proposition \ref{shadowlemmab} gives us
  $\op{f}(x)-\gamma_A\V_{k+1}(f)/p = -\gamma_A\V_{k+1}(f)/p \leq
  \op{f}(y)$ so that $-\op{f}(y)\le LA_{k+1}/p$.  Combining this with
  the fact that $\|h\|_\infty\le\|h\|_A <
  \varepsilon\sigma$ gives $\beta < L A_{k+1}/p + \varepsilon\sigma$. Since
  $q=\hat f-\varepsilon g+h+\beta$ and the first two terms are non-positive,
  we see that
  \begin{equation}\label{eq:qbound}
  \begin{split}
  h(\omega)+\beta&<\frac{L A_{k+1}}{p}+2\varepsilon\sigma\text{ for all
  $\omega\in\Omega$; and}\\
  q(\omega)  &< \frac{L A_{k+1}}{p}+2\varepsilon\sigma\text{ for all
  $\omega\in\Omega$.}
  \end{split}
  \end{equation}

  Let $q^{(n)}$ be the co-cycle
  $q^{(n)}(z) = q(T^{n-1}z)+q(T^{n-2}z)+\cdots+q(z)$,
  and note that if $n>m$, $q^{(n)}(z)-q^{(m)}(z) = q^{(n-m)}(T^mz)$.

  We know by Proposition \ref{fixedpoint} that there exists $q^*$, a fixed
  point of $\Phi_q$. Let $z\in\Omega$ be arbitrary. We know there exists some
  symbol $a_1$ such that $q^*(z) = q(a_1z)+q^*(a_1z)$.  Iterating this
  process, we may find an infinite sequence of preimages $(a_i)$ such that for
  any $n>0$,
  \begin{equation}\label{eq:optpreim}
  \begin{split}
    q^*(z) &= q(a_1z)+q(a_2a_1z)+\cdots+q(a_n\cdots
    a_1z)+q^*(a_n\cdots a_1z)\\
    &= q^{(n)}(a_n\cdots a_1z)+q^*(a_n\cdots a_1z).
  \end{split}
  \end{equation}

  Fix any such preimage infinite sequence $(a_i)$.  We will now identify
  a (possibly finite) sequence of times,
  $(t_n)$, by the following recursive
  procedure: For a time $t$, define $\omega_t = a_{t}a_{t-1}\cdots a_1
  z$.  Let $t_0$ be the smallest number (if it exists)
  such that $d(\omega_{t_0},\Orb
  y) > 2^{-(k+1)}$.  Given $t_n$, let $t_{n+1}>t_n$ be the next
  smallest number (again, if it exists)
  so that $d(\omega_{t_{n+1}},\Orb y) > 2^{-(k+1)}$. Our goal is to show
  that the length of the sequence is finite. From this it follows that the
  preimages $\omega_t$ accumulate to $\Orb y$. It will
  then follow that the periodic orbit measure
  supported on $\Orb y$ is the unique maximizing measure.

  Since $2^{-l} \geq 2^{-(k-1)}$ (and so $2^{-l}/4\geq 2^{-(k+1)}$),
  for times strictly between $t_n$ and $t_{n-1}$, the In Order Lemma
  (Lemma \ref{inorderlemma}) gives that we $2^{-(k+1)}$-shadow $\Orb y$.

  Suppose $t_n-t_{n-1} > 1$ and let $y'\in \Orb y$ be the point that
  is $2^{-(k+1)}$-shadowed by $\omega_{t_n}$ for $t_n-t_{n-1} - 1$
  steps (that is $d(T^i\omega_{t_n},T^iy)\le 2^{-(k+1)}$ for $0<i<t_n-t_{n-1}$).
  Summing along this segment, the Parallel Orbit Lemma
  (Lemma \ref{parallelorbitlemma}) gives us
  \[
  \sum_{0<i<t_n-t_{n-1}} \left[ q(T^i\omega_{t_n}) - q(T^{i}y') \right]
  \leq \V_{k+1}(q)\leq L A_{k+1}.
  \]
  so that
  \begin{equation*}
    \sum_{0<i<t_n-t_{n-1}} q(T^i\omega_{t_n})
    \leq L A_{k+1} + \sum_{0<i<t_n-t_{n-1}}q(T^{i}y')
  \end{equation*}

  Grouping $\sum_{0<i<t_n-t_{n-1}}q(T^{i}y')$ in blocks
  of length $p$ together with at most $p-1$ singleton terms and using
  \eqref{eq:qbound}, we see
  \begin{equation*}
    \sum_{0<i<t_n-t_{n-1}} q(T^i\omega_{t_n})
    \leq L A_{k+1} + mp\op{q}(y) + (p-1)\left
      (LA_{k+1}/p+2\varepsilon\sigma \right) ,
  \end{equation*}
  where $m$ is the integer part of $(t_n-t_{n-1}-1)/p$.  Since
  $\op{q}(y) \leq 0$, we simplify to get
  \begin{equation}
    \label{eq:inorderbound}
    \sum_{0<i<t_n-t_{n-1}} q(T^i\omega_{t_n}) \leq 2LA_{k+1} + 2(p-1)\varepsilon\sigma.
  \end{equation}
  Notice that this equation holds also (trivially) if $t_n=t_{n-1}+1$.
  We now evaluate $q(\omega_{t_n})$:
  \[
  q(\omega_{t_n} ) = \hat f(\omega_{t_n} ) -\varepsilon g(\omega_{t_n} ) + h(\omega_{t_n} )
  +\beta.
  \]
  By construction we have $d(\omega_{t_n} , \Orb y) \geq  2^{-k}$ so that
  $g(\omega_{t_n})\ge A_k$. Using \eqref{eq:qbound} again and the fact
  that $\hat f \leq 0$ we have
  \begin{equation}
    \label{eq:outoforderbound}
    q(\omega_{t_n} ) \leq -\varepsilon A_k + \frac{L A_{k+1}}{p} + 2\varepsilon\sigma.
  \end{equation}

  Combining equations (\ref{eq:inorderbound}) and
  (\ref{eq:outoforderbound}) we get
  \[
  q^{(t_n-t_{n-1})}(\omega_{t_n}) \leq -\varepsilon A_k + 3L A_{k+1} + 2p\varepsilon\sigma,
  \]
  and so for $\sigma \leq A_k/(4p)$ we have
  \[
  q^{(t_n-t_{n-1})}(\omega_{t_n}) \leq -\frac{\varepsilon}{2} A_k + 3L A_{k+1}.
  \]
  Since $L$ only depends on $(A_n)$ and $\|f\|_A$, our assumption that
  $A_{k+1}/A_k\to 0$ ensures that there exists a $k$ such that
  $\alpha=\frac{\varepsilon}{2} A_k - 3L A_{k+1}> 0$. Fix this $k$ and fix
  $\sigma=A_k/(4p)$. Let $(x)_i^{j-1}$ be the minimal recurrence segment
  identified in the proof and $y$ be the corresponding periodic orbit.
  This fixes the open ball $Q$ whose centre is at a distance $\varepsilon$
  from $\hat f$.

  We have shown that for any function in $Q$, its normalized version $q$
  satisfies
  $q^{(t_i-t_{i-1})}(\omega_{t_i}) < -\alpha$.
  Expanding using \eqref{eq:optpreim} now gives
  \[
  q^*(\omega_{t_0})-q^*(\omega_{t_n} ) = q^{(t_n-t_0)}(\omega_{t_n} ) = \sum_{i=1}^n
  q^{(t_i-t_{i-1})}(\omega_{t_i} ) \leq -n\alpha.
  \]
  But $q^*$ is a bounded function and so the number of terms in the sequence $(t_n)$
  is finite.

  Since $z$ was chosen arbitrarily, this is sufficient to show the
  periodic orbit measure supported on $\Orb y$ uniquely optimizes $q$. If not,
  then there would be points $z$ and preimage sequences $(a_i)$ satisfying
  \eqref{eq:optpreim} that do
  not eventually follow $\Orb y$, and so $(t_n)$ would be
  infinite.
  \end{proof}

  Theorem \ref{mainresult} proves both (a) that a function optimized
  by an aperiodic point can be perturbed to be optimized by a periodic
  point and (b) that a function optimized by periodic point can be
  perturbed to lie in an open set of functions optimized by the same
  periodic point.  Following the methods of Yuan and Hunt in
  \cite{yuanhunt}, one can prove (b) in the general context of
  $A$-norm spaces (dropping the assumption that $A_{n+1}/A_{n}\to 0$
  entirely).

  \begin{acknowledgements}
    We would like to thank the referee for a careful reading and very
    useful suggestions.
  \end{acknowledgements}

\footnotesize

\bibliographystyle{abbrv}
\bibliography{supercts}

\end{document}